\documentclass[11pt]{amsart}
\usepackage{a4wide,amssymb,hyperref}

\newcommand{\w}{\omega}
\newcommand{\Ra}{\Rightarrow}
\newcommand{\IR}{\mathbb R}
\newcommand{\I}{\mathcal I}
\newcommand{\U}{\mathcal U}
\newcommand{\pr}{\mathrm{pr}}

\newtheorem{theorem}{Theorem}[section]

\newtheorem{proposition}[theorem]{Proposition}
\newtheorem{example}[theorem]{Example}
\newtheorem{corollary}[theorem]{Corollary}
\newtheorem{lemma}[theorem]{Lemma}
\newtheorem{conjecture}[theorem]{Conjecture}
\newtheorem{question}[theorem]{Question}

\theoremstyle{definition}

\newtheorem{remark}[theorem]{Remark}

\title{Banalytic spaces and characterization of Polish groups}
\author{Taras Banakh and Alex Ravsky}
\address{T.Banakh: Institute of Mathematics, Jan Kocahnowski University in Kielce (Poland) and Ivan Franko National University of Lviv (Ukraine)}
\email{t.o.banakh@gmail.com}
\address{A.Ravsky: Pidstryhach Institute for Applied Problems of Mechanics and Mathematics National Academy of Sciences of Ukraine, Lviv, Ukraine}
\email{alexander.ravsky@uni-wuerzburg.de}

\begin{document}
\begin{abstract} A topological space is defined to be banalytic (resp. analytic) if it is the image of a Polish space under a Borel (resp. continuous) map.  A regular topological space is analytic if and only if it is banalytic and cosmic. Each (regular) banalytic space has countable spread (and under PFA is hereditarily Lindel\"of). Applying banalytic spaces to topological groups, we prove that for a Baire topological group $X$ the following conditions are equivalent: (1) $X$ is Polish, (2) $X$ is analytic, (3) $X$ is banalytic and cosmic, (4) $X$ is banalytic and has countable pseudocharacter. Under PFA the conditions (1)--(4) are equivalent to the banalycity of $X$. The conditions (1)--(3) remain equivalent for any Baire  semitopological group.
\end{abstract}

\maketitle

\section{Introduction}

In this space we introduce banalytic spaces, which are generalizations of analytic spaces, and then apply analytic and banalytic spaces to characterizing Polish topological groups in the classes of topological, semitopological and paratopological groups.

A topological space is
\begin{itemize}
\item {\em Polish} if it is homeomorphic to a separable complete metric space;
\item {\em cosmic} if $X$ is regular and is a continuous image of a separable metrizable space;
\item {\em analytic} if it is a continuous image of a Polish space;
\item {\em banalytic} if it is a Borel image of a Polish space, i.e., the image of a Polish space under a surjective Borel function.
\end{itemize}

We recall that a function $f:X\to Y$ between topological spaces is {\em Borel} if for any open set $U\subset Y$ its preimage $f^{-1}(U)$ is a Borel subset of $X$. In this case the preimage $f^{-1}(B)$ of any Borel set $B\subset Y$ is a Borel subset of $X$.

Analytic spaces are well-studied in the Descriptive Set Theory  in the frameworks of metrizable spaces \cite{Ke}, \cite{Most} and of general topological spaces  \cite{Analytic}, \cite{KKP}. On the other hand, the notion of a banalytic space seems to be new.

It is clear that each analytic space is banalytic. The Sorgenfrey line is banalytic but fails to have a countable network and hence is not analytic (see Example~\ref{e:S} for more details).

Therefore, banalytic spaces form a class of topological spaces that includes all analytic spaces. In Corollary~\ref{c:ban} we shall prove that a regular topological space is analytic if and only if it is cosmic and banalytic. In Theorem~\ref{t:spread} we prove that (regular) banalytic $T_{\mathsf{Borel}}$-spaces have countable spread (and under PFA are hereditarily Lindel\"of). Here PFA is the abbreviation for the Proper Forcing Axiom, see \cite{Baum}, \cite{Moors}. A topological space $X$ is defined to have {\em countable spread} (resp. {\em countable extent}) if each (closed) discrete subspace of $X$ is at most countable.

A topological space $X$ is called a {\em $T_{\mathsf{Borel}}$-space} if each singleton $\{x\}\subset X$ is a Borel subset of $X$.  This separation axiom was introduced by Harley and McNultey \cite{HMc} and was also studied in \cite{BB}. By \cite{HMc}, for any topological space we have the implications: $T_1\Ra T_{\mathsf{Borel}}\Ra T_0$. 

Applying (b)analytic spaces to topological groups, in Section~\ref{s:G} we prove the following characterization of Polish groups.

\begin{theorem}\label{t:G} For a topological group $X$ the following conditions are equivalent:
\begin{enumerate}
\item $X$ is Polish;
\item $X$ is a Baire analytic $T_0$-space;
\item $X$ is Baire, banalytic and cosmic;
\item $X$ is Baire, banalytic and has countable pseudocharacter.
\end{enumerate}
Under PFA, the conditions \textup{(1)--(4)} are equivalent to
\begin{itemize}
\item[(5)] $X$ is a Baire banalytic $T_0$-space.
\end{itemize}
\end{theorem}

We recall that a topological space is {\em Baire} if for any sequence $(U_n)_{n\in\w}$ of dense open subsets in $X$ the intersection $\bigcap_{n\in\w}U_n$ is dense in $X$. A topological space $X$ has {\em countable pseudocharacter} if each singleton $\{x\}\subset X$ is a $G_\delta$-set in $X$. In this case $X$ is a $T_{\mathsf{Borel}}$-space.

For topological groups whose topology is generated by an invariant metric the equivalence $(1)\Leftrightarrow(2)$ in Theorem~\ref{t:G} was proved by Christensen in \cite[5.4]{Chris}.

In Sections~\ref{s:S} and \ref{s:P} we shall apply Theorem~\ref{t:G} to characterize Polish topological groups among semitopological and paratopological groups.

A group $X$ endowed with a topology is called
\begin{itemize}
\item a {\em semitopological group} if the group operation $X\times X\to X$, $(x,y)\mapsto xy$, is separately continuous;
\item a {\em paratopological group} if the group operation $X\times X\to X$, $(x,y)\mapsto xy$, is continuous;
%\item a {\em quasitopological group} if it is a semitopological group with continuous inversion $X\to X$, $x\mapsto x^{-1}$;
\item a {\em topological group} if $X$ is a paratopological group with continuous inversion $X\to X$, $x\mapsto x^{-1}$.
\end{itemize}

The following characterization is a ``semitopological'' version of  Theorem~\ref{t:G}.

\begin{theorem}\label{t:S} For a semitopological group $X$ the following conditions are equvalent:
\begin{enumerate}
\item $X$ is a Polish topological group;
\item $X$ is a banalytic cosmic Baire space;
\item $X$ is banalytic, Baire, has countable pseudocharacter and is topologically isomorphic to a subgroup of the Tychonoff product of cosmic semitopological groups.
\end{enumerate}
Under PFA the conditions  \textup{(1)--(3)} are equivalent to
\begin{itemize}
\item[(4)] $X$ is banalytic, Baire, and $X$ is topologically isomorphic to a subgroup of the Tychonoff product of cosmic semitopological groups.
\end{itemize}
\end{theorem}

Finally, we present a characterization of Polish topological groups among paratopological groups. Following Tkachenko \cite{Tk}, we define a paratopological group $X$ to be
\begin{itemize}
\item  {\em $\w$-narrow} if for any neighborhood $U\subset X$ of the unit there exists a countable set $C\subset X$ such that $X=C\cdot U=U\cdot C$;
\item {\em totally $\w$-narrow} if for any neighborhood $U\subset X$ of the unit there exists a countable set $C\subset X$ such that $X=C\cdot (U\cap U^{-1})$.
\end{itemize}
%By Theorem 2.8 of \cite{Tk}, each countably Hausdorff totally $\w$-narrowparatopological group is topologically isomorphic to a subgroup of the Tychonoff product of second-countable Hausdorff paratopological groups. This result allows us to apply Theorems~\ref{t:G}, \ref{t:S} and prove the following characterization.

\begin{theorem}\label{t:P} For a paratopological group $X$  the following conditions are equivalent:
\begin{enumerate}
\item $X$ is a Polish topological group;
\item $X$ is a banalytic cosmic Baire space;
\item $X$ is a banalytic Baire space with countable pseudocharacter and the paratopological group $X$  is totally $\w$-narrow;
\item $X$ is a banalytic Baire $T_1$-space with countable pseudocharacter and the square $X\times X$ has countable extent;
\item $X$ is a banalytic Baire space with countable pseudocharacter and the square $X\times X$ has countable spread;
\item $X$ is a Baire space with countable pseudocharacter and the square $X\times X$ is banalytic.
\end{enumerate}
Under PFA the conditions  \textup{(1)--(6)} are equivalent to
\begin{itemize}
\item[(7)] $X$ is a Baire banalytic $T_1$-space and the square $X\times X$ has countable extent.
\item[(8)] $X$ is a Baire banalytic $T_0$-space and the square $X\times X$ has countable spread.
\item[(9)] $X$ is Baire $T_{\mathsf{Borel}}$-space with banalytic square $X\times X$.
\end{itemize}
\end{theorem}

We do not know if the assumption of PFA can be removed from Theorems~\ref{t:G}, \ref{t:S}, \ref{t:P}. This can be done if the following conjecture is true.

\begin{conjecture} A Hausdorff topological group is Polish if and only if it is Baire and banalytic.
\end{conjecture}

Theorems~\ref{t:G}, \ref{t:S}, \ref{t:P} will be proved in Sections~\ref{s:G}, \ref{s:S}, \ref{s:P} after some preliminary work, made in Sections~\ref{s:ban}, \ref{s:ab}, \ref{s:fac}. In Section~\ref{s:ban} we establish some elementary properties of banalytic spaces, in Section~\ref{s:ab} we study the interplay between analytic and banalytic spaces and in Section~\ref{s:fac} we prove a factorization theorem for Baire banalytic spaces.

\section{Some properties of banalytic spaces}\label{s:ban}

In this section we establish some elementary properties of banalytic spaces.

\begin{proposition}\label{p:im} A Borel image of a banalytic space is banalytic.
\end{proposition}

\begin{proof} Given a surjective Borel function $f:X\to Y$ from a banalytic space $X$ onto a topological space $Y$, we should prove that the space $Y$ is banalytic. Find a surjective Borel function $g:P\to X$ defined on a Polish space $X$ and observe that the composition $g\circ f:P\to Y$ is a surjective Borel function, witnessing that the space $Y$ is banalytic.
\end{proof}

\begin{proposition}\label{p:sub} Each Borel subspace of a banalytic space is banalytic.
\end{proposition}

\begin{proof} Let $X$ be a banalytic space and $B\subset X$ be a Borel subset of $X$. Being banalytic, the space $X$ is the image of a Polish space $P$ under a  Borel function $f:P\to X$. Then the preimage $f^{-1}(B)$ is Borel and by \cite[13.1]{Ke}, we can find a continuous bijective map $g:Z\to P$ from a suitable Polish space $Z$ such that the preimage $g^{-1}(f^{-1}(B))$ is clopen in $Z$ and hence is a Polish space. Then the map $h=f\circ g{\restriction}g^{-1}(Z):g^{-1}(B)\to B$ is Borel, witnessing that the space $B$ is banalytic.
\end{proof}

\begin{proposition} The countable product $X=\prod_{n\in\w}X_n$ of banalytic spaces is banalytic if $X$ is hereditary Lindel\"of.
\end{proposition}

\begin{proof} For every $n\in\w$ fix a surjective Borel function $f_n:P_n\to X_n$. Consider the Polish space $P=\prod_{n\in\w}P_n$ and the surjective function $$f:P\to X,\;\;f:(x_n)_{n\in\w}\mapsto(f_n(x_n))_{n\in\w}.$$
Assuming that the space $\prod_{n\in\w}X_n$ is hereditarily Lindel\"of, we shall prove that the function $f$ is Borel.
Given an open subset $U\subset X$, we shall show that the preimage $f^{-1}(U)$ is a Borel subset of $P$. By definition of the product topology on $X=\prod_{n\in\w}X_n$, for any $x\in U$ there exists a number $n_x\subset \w$ and open sets $U_i\subset X_i$ for $i<n_x$ such that $$x\in O_x:=\prod_{i<n_x}U_i\times \prod_{i\ge n_x}X_i\subset U.$$Taking into account that the functions $f_n$, $n\in\w$, are Borel, we can see that the preimage $f^{-1}(O_x)=\prod_{u<n_x}f_i^{-1}(U_i)\times\prod_{i\ge n_x}P_i$ is a Borel subset of the Polish space $P$.

Since the space $X$ is hereditarily Lindel\"of, there exists a countable subset $C\subset U$ such that $U=\bigcup_{x\in C}O_x$. Then $f^{-1}(U)=\bigcup_{x\in C}f^{-1}(O_x)$ is Borel in $P$, being the countable union of Borel subsets $f^{-1}(O_x)$, $x\in C$.
\end{proof}

\begin{lemma}\label{l:disc} Any discrete subspace $D$ of a $T_{\mathsf{Borel}}$-space $X$ is a Borel subset of $X$. More precisely, $D$ is a $G_\delta$-set in its closure $\bar D$.
\end{lemma}

\begin{proof} Replacing $X$ by the closure of $D$, we can assume that the discrete subspace $D$ is dense in $X$. Since the space $D$ is discrete, the open set $O_x=X\setminus \overline{D\setminus\{x\}}$ is an open neighborhood of $x$ such that $O_x\cap D=\{x\}$. The density of $D$ in $X$ implies that $O_x\subset\overline{\{x\}}$. Moreover, for any distinct points $x,y\in D$ we have $O_x\cap O_y\cap D=\{x\}\cap\{y\}=\emptyset$ and hence $O_x\cap O_y=\emptyset$ by the density of the set $D$ in $X$.

By Theorem~2.1 of \cite{BB}, the separation axiom $T_{\mathsf{Borel}}$ implies that each singleton $\{x\}\subset X$ is a $G_\delta$-set in its closure $\overline{\{x\}}$. Consequently, for every $x\in D$ we can find a decreasing sequence $(U_n(x))_{n\in\w}$ of open sets in $O_x\subset\overline{\{x\}}$ such that $\{x\}=\bigcap_{n\in\w}U_n(x)$.  For every $n\in\w$ consider the open set $U_n=\bigcup_{x\in D}U_n(x)$ in $X$.
Taking into account that the family $(O_x)_{x\in D}$ is disjoint, we see that $\bigcap_{n\in\w}U_n=\bigcap_{n\in\w}\bigcup_{x\in D}U_n(x)=\bigcup_{x\in D}\bigcap_{n\in\w}U_n(x)=\bigcup_{x\in D}\{x\}=D$. Therefore, $D$ is a $G_\delta$-set in $X=\bar D$.
\end{proof}

\begin{theorem}\label{t:spread} Each banalytic $T_{\mathsf{Borel}}$-space has countable spread.
\end{theorem}

\begin{proof} To derive a contradiction, assume that some banalytic $T_{\mathsf{Borel}}$-space $X$ contains an uncountable discrete subspace $D$, which is a Borel subset of $X$ according to Lemma~\ref{l:disc}. 
By Proposition~\ref{p:sub}, the discrete space $D$ is banalytic and hence admits a surjective Borel function $f:P\to D$, defined on some Polish space $P$. Consider the $\sigma$-ideal $$\mathcal I:=\{I\subset P:|f(I)|\le\omega\}$$ of subsets of $P$ and observe that $\mathcal A:=\{f^{-1}(x)\}_{x\in D}\subset \I$ is a disjoint (and hence point-finite) family in $P$ whose union $\bigcup\mathcal A=P$ does not belong to the ideal $\mathcal I$. Since $X$ is a $T_{\mathsf{Borel}}$-space, each singleton $\{x\}\subset X$ is Borel in $X$ and its preimage $f^{-1}(x)$ is a Borel subset of $X$. This implies that each set $I\in\I$ is contained in the Borel set $f^{-1}(f(I))$ and hence the ideal $\mathcal I$ has a Borel base.

By \cite{Buk} or \cite{Pol4}, there exists a subfamily $\mathcal A'\subset\mathcal A$ whose union $\bigcup\mathcal A'$ does not belong to the smallest $\sigma$-algebra that contains all Borel sets and all sets in the $\I$. In particular, the set $\bigcup\mathcal A'$ is not Borel in $P$. On the other hand, $\bigcup\mathcal A'=f^{-1}(E)$ is Borel, being the preimage of the open set $E=f(\bigcup\mathcal A')$ in $D$. This is a desired contradiction completing the proof of Theorem~\ref{t:spread}.
\end{proof}

Under PFA, the Proper Forcing Axiom, we can prove a bit more.

\begin{corollary}\label{c:PFA} Under PFA, each Hausdorff banalytic space has countable pseudocharacter.
\end{corollary}

This corollary follows from Theorem~\ref{t:spread} and Corollary 8.11 of \cite{Todor} (saying that that under PFA each Hausdorff space with countable spread has countable pseudocharacter).

\begin{remark} In \cite{HJ} Hajnal and Juh\'asz constructed a CH-example of a hereditarily separable subgroup $G$ of the compact topological group $\mathbb Z_2^{\omega_1}$ such that $G$ is not Lindel\"of and has uncountable pseudocharacter (moreover, $G$ has non-empty intersection with any non-empty $G_\delta$-subset of $\mathbb Z_2^{\omega_1}$). This example shows that Corollary 8.11 of \cite{Todor} used in the proof of Corollary~\ref{c:PFA} does not hold under CH.
\end{remark}

\begin{question} Is Corollary~\ref{c:PFA} true under CH?
\end{question}

\section{Analytic versus banalytic spaces}\label{s:ab}

It is clear that each analytic space is banalytic. The converse is true for topological spaces possessing a countable Borel network.

A family $\mathcal N$ of subsets of a topological space $X$ is called a {\em network} if for any open set $U\subset X$ and point $x\in U$ there exists a set $N\in\mathcal N$ such that $x\in N\subset U$. A network $\mathcal N$ is called {\em Borel} (resp. {\em closed, open}) if each set $N\in\mathcal N$ is Borel (resp. closed, open) in $X$.

Open networks are exactly bases of the topology. For any network $\mathcal N$ in a regular topological space, the family $\{\bar N:N\in\mathcal N\}$ is a closed network for $X$. So, a regular topological space has a countable network if and only if it has a countable closed network.

\begin{example} Let $X$ be a set of cardinality $\w<|X|\le\mathfrak c$, endowed with the topology $\tau=\{\emptyset\}\cup\{X\setminus F:F$ is a finite set$\}$. The topological space $(X,\tau)$ has a countable network but does not have a countable Borel network.
\end{example}

\begin{proof}  Since $|X|\le\mathfrak c$, there exists a metrizable separable topology $\sigma$ on $X$. This topology has a countable base $\mathcal B\subset \sigma$. Taking into account that $\tau\subset\sigma$, we see that $\mathcal B$ is a countable network of the topology $\tau$.

Now assume that $\tau$ has a countable Borel network $\mathcal N$. Then each set $N\in\mathcal N$, being Borel in $(X,\tau)$ is either countable or has countable complement. So, $\mathcal N=\mathcal N_0\cup\mathcal N_1$ where $\mathcal N_0=\{N\in\mathcal N:|N|\le\w\}$ and $\mathcal N_1=\{N\in\mathcal N:|X\setminus N|\le\w\}$. Choose two distinct points $x,y\notin X\setminus (\bigcup\mathcal N_0\cup\bigcup_{N\in\mathcal N_1}X\setminus N)$ and observe that for the  neighborhood $U=X\setminus \{y\}$ of the point $x$ there is no set $N\in\mathcal N$ with $x\in N\subset U$. This means that $\mathcal N$ is not a network.
\end{proof}

\begin{question} Is there a Hausdorff topological space that has a countable network but does not have a countable Borel network?
\end{question}

\begin{lemma}\label{l:Bn} Let $f:P\to X$ be a surjective Borel function from a Polish space $P$ onto a topological space $X$. If the space $X$ has a countable Borel network, then there exists a continuous bijective map $g:Z\to P$ from a Polish space $Z$ such that the map $f\circ g:Z\to X$ is continuous.
\end{lemma}

\begin{proof}Assume that the space $X$ has a countable Borel network $\mathcal N$. By the Borel property of the map $f$, for every $N\in\mathcal N$ the  preimage $f^{-1}(N)$ is a Borel subset of $P$. By \cite[13.5]{Ke}, for the countable family $\{f^{-1}(N)\}_{N\in\mathcal N}$ of Borel subsets of $P$, there exists a bijective continuous map $g:Z\to P$ from a Polish space $Z$ such that for any $N\in\mathcal N$ the Borel set $f^{-1}(N)\subset P$ has clopen preimage $g^{-1}(f^{-1}(N))$ in $Z$. Consider the map $\varphi=f\circ g:Z\to X$ and observe that for every $N\in\mathcal N$ the preimage $\varphi^{-1}(N)=g^{-1}(f^{-1}(N))$ is open in $Z$. Since each open subset $U\subset X$ coincides with the union $\bigcup\mathcal N_U$ of the subfamily $\mathcal N_U:=\{N\in\mathcal N:N\subset U\}$, the preimage $\varphi^{-1}(U)=\bigcup_{N\in\mathcal N_U}\varphi^{-1}(N)$ is open in $Z$, which means that the (surjective) map $\varphi:Z\to X$ is  continuous.
\end{proof}

Lemma~\ref{l:Bn} implies the following characterization.

\begin{corollary}\label{c:ban1} A topological space $X$ with a countable Borel network if analytic if and only if it is banalytic.
\end{corollary}

A topological space $X$ is {\em cosmic} if it is regular and has a countable network.
Since each regular space with a countable network has a countable closed network, Corollary~\ref{c:ban1} implies the following characterization.

\begin{corollary}\label{c:ban} A regular space $X$ is analytic if and only if $X$ is cosmic and banalytic.
\end{corollary}

\begin{example}\label{e:S} The Sorgenfrey line $\mathbb S$ is banalytic but not analytic.
\end{example}

\begin{proof} We recall that the {\em Sorgenfery line} is the real line endowed with the topology, generated by the base consisting of the half-intervals $[a,b)$ where $a<b$. It is well-known \cite[3.8.14]{Eng} that the Sorgenfrey line is hereditarily Lindel\"of, which implies that each open set $U\subset \mathbb S$ is a countable union of half-intervals and hence is a Borel (more precisely, $F_\sigma$-set) subset of the real line endowed with the Euclidean topology. This means that the identity map $\IR\to\mathbb S$ is Borel and hence $\mathbb S$ is banalytic. On the other hand, the Sorgenfrey line does not have countable network and hence cannot be analytic according to Corollary~\ref{c:ban}.
\end{proof}

\section{A factorization theorem for Baire banalytic spaces}\label{s:fac}

In this section we shall prove an important factorization theorem for Baire banalytic spaces. This theorem will be applied in Sections~\ref{s:G}, \ref{s:S}, \ref{s:P} to studying banalytic topological, semitopological, and paratopological groups.

%Let us recall that a topological space $X$ is
%\begin{itemize}
%\item {\em Baire} if for any countable sequecne $(U_n)_{n\in\w}$ of open dense sets in $X$ the intersection $\bigcap_{n\in\w}U_n$ is dense in $X$;
%\item {\em meager} if $X$ can be written as the countable union of closed sets with empty interior.
%\end{itemize}
%It is known that a topological space is Baire if and only if it contains no non-empty %open meager subspace.

A subset $A$ of a topological space $X$ is called {\em meager in} $X$ if $A$ can be written as the countable union of nowhere dense subsets in $X$. For any subset $A$ of a topological space let $\U_A$ be the family of all open subsets $U\subset X$ such that $U\cap A$ is meager in $X$. Then the set $A^\bullet=X\setminus\U_A$ is regular closed in $X$. Moreover,  $A$ is not meager in $X$ if and only if $A^\bullet\ne\emptyset$.

\begin{theorem}\label{t:fac} Let $X_\alpha$, $\alpha\in A$, be topological spaces with a countable Borel network and $X$ be a Baire banalytic subspace of the Tychonoff product $\prod_{\alpha\in A}X_\alpha$. Then each countable set $C\subset A$ is contained in a countable subset $D\subset A$ such that the projection $X_D$ of $X$ onto the $D$-face $\prod_{\alpha\in D}X_\alpha$ of\/ $\prod_{\alpha\in A}X_\alpha$ is a Baire analytic space.
\end{theorem}

\begin{proof} Being banalytic, the space $X$ admits a surjective Borel map $f:P\to X$, defined on a Polish space $P$. Let $\tau_0$ be the topology of the Polish space $P$ and $\mathcal B_0$ be any countable base of the topology $\tau_0$. Let $\mathcal B_0^\bullet:=\{B\in\mathcal B_0: f(B)^\bullet\ne\emptyset\}$. By the definition of $\mathcal B_0^\bullet$, for any $B\in\mathcal B_0^\bullet$ the set $f(B)^\bullet$ has non-empty interior $f(B)^{\bullet\circ}$ in $X$.

 For any subset $D\subset A$ let $\pr_D:X\to\prod_{\alpha\in D}X_\alpha$ be the projection of $X$ onto the $D$-face $\prod_{\alpha\in D}X_\alpha$ of $\prod_{\alpha\in A}X_\alpha$, and let $X_D:=\pr_D(X)\subset\prod_{\alpha\in D}X_\alpha$.

Given any countable set $C\subset A$, find a countable set $D_0\subset A$ such that $C\subset D_0$ and for every $B\in\mathcal B_0^\bullet$ the open non-empty set $f(B)^{\bullet\circ}$ contains the preimage $\pr_{D_0}^{-1}(U_B)$ of some non-empty open set $U_B\subset X_{D_0}$. Since the space $X_{D_0}$ has a countable Borel network, we can apply Lemma~\ref{l:Bn} and find a Polish topology $\tau_1$ on $P$ such that $\tau_0\subset\tau_1$ and the map $\pr_{D_0}\circ f:(P,\tau_1)\to X_{D_0}$ is continuous. Let $\mathcal B_1$ be any countable base of the topology $\tau_1$.

Proceeding by induction, for every $n\in\mathbb N$ we can construct a Polish topology  $\tau_n$ on $P$, a countable base $\mathcal B_n$ of the topology $\tau_n$ and a countable subset $D_n\subset A$ such that the following conditions are satisfied:
\begin{itemize}
\item[$(a_n)$] $\tau_{n-1}\subset\tau_n$ and $D_{n-1}\subset D_n$;
\item[$(b_n)$] for any $B\in\mathcal B_n$ with $f(B)^\bullet\ne\emptyset$, the set $f(B)^\bullet$ contains the preimage $\pr_{D_n}^{-1}(U_B)$ of some non-empty open set $U_B\subset X_{D_n}$;
\item[$(c_n)$] the map $\pr_{D_n}\circ f:(P,\tau_{n+1})\to X_{D_n}$ is continuous.
\end{itemize}
 Let $D=\bigcup_{n\in\w}D_n$ and $\tau$ be the topology on $P$ generated by the base $\bigcup_{n\in\w}\mathcal B_n$. By \cite[13.3]{Ke}, the topology $\tau$ is Polish. The inductive conditions $(c_n)$, $n\in\mathbb N$, ensure that the function $g=\pr_D\circ f:(P,\tau)\to X_D$ is continuous.  Then the space $X_D$ is analytic.

 It remains to prove that the space $X_D$ is Baire. In the opposite case, we could find a non-empty open meager subset $W\subset X_D$. Then $W\subset \bigcup_{n\in\w}M_n$ for an increasing sequence $(M_n)_{n\in\w}$ of closed nowhere dense sets in $X_D$. By the continuity of the projection $\pr_D:X\to X_D$, the preimage $\pr_D^{-1}(W)$ is an open non-empty subset of $X$. Since $X$ is Baire, the open subset $\pr_D^{-1}(W)$ is not meager in $X$. By the continuity of the map $g=\pr_D\circ f:(P,\tau)\to X_D$, the set $V:=g^{-1}(W)$ is open in the Polish space $(P,\tau)$.

Let $\U$ be the family of all open sets $U\in \tau$ such that $f(U)$ is meager in $X$. Since the Polish space $(P,\tau)$ is hereditarily Lindel\"of, the union $U=\bigcup\U$ belongs to $\U$ and has meager image $f(U)$ in $X$. It follows that the set $F=P\setminus U$ is closed with respect to the topology $\tau$.

Taking into account that $f(U)$ is meager in $X$ and $\pr_D^{-1}(W)=f(V)$ is not, we conclude that $f(V)\not\subset f(U)$ and hence the intersection $V\cap F$ is not empty. Since $V\cap F\subset \bigcup_{n\in\w}g^{-1}(M_n)$ we can apply the Baire Theorem and find $n\in\w$ such that the set $V\cap F\cap g^{-1}(M_n)$ has non-empty interior in $V\cap F$. Consequently, there exist  a point $z\in V\cap F$ and a neighborhood $O_z\in\tau$ of $z$ such that $O_z\cap V\cap F\subset g^{-1}(M_n)$. Since the topology $\tau$ is generated by the base $\bigcup_{k\in\w}\mathcal B_k$, we can assume that $O_z\in\bigcup_{k\in\w}\mathcal B_k$ and $O_z\subset V$. Find $k\in\w$ such that $O_z\in\mathcal B_k$. It follows from $y\in F$ that $f(O_z)$ is not meager in $X$ and hence $f(O_z)^\bullet\ne\emptyset$. Then the inductive condition $(b_k)$ ensures that $f(O_z)^\bullet$ contains the preimage $\pr_{D_k}^{-1}(G)$ of some open set $G\subset X_{D_k}$ and hence the set $\pr_D(f(O_z)^\bullet)$ has non-empty interior in $X_D$.

Taking into account that the set $f(O_z\setminus F)\subset f(U)$ is  meager in $X$, we conclude that $f(O_z)^\bullet=f(O_z\cap F)^\bullet\subset \overline{f(O_z\cap F)}$ and hence
 $$\pr_D(f(O_z)^\bullet)\subset \pr_D(\overline{f(O_z\cap F)})\subset \overline{\pr_D\circ f(O_z\cap F)}=\overline{g(O_z\cap F)}\subset\overline{M_k}=M_k,$$which implies that the set $M_k$ has non-empty interior in $X_D$ and this is a desired contradiction, proving that the space $X_D$ is Baire.
\end{proof}

\section{Proof of Theorem~\ref{t:G}}\label{s:G}

In this section we prove Theorem~\ref{t:G}. The non-trivial implications $(2,3,4,5)\Ra(1)$ of this theorem are proved in Lemmas~\ref{l:G1}, \ref{l:G2}, \ref{l:G3} and \ref{l:G4}, respectively. We assume that all topological groups considered in this section satisfy the separation axiom $T_0$ and hence are regular, see \cite[1.3.14]{AT}.

\begin{lemma}\label{l:Bm} Each Baire topogical group $X$ with countable network is metrizable and separable.
\end{lemma}

\begin{proof}  Let $\mathcal N$ be a countable network for the space $X$. Since the space $X$ is regular, we can replace each set $N\in\mathcal N$ by its closure $\bar N$ in $X$ and assume that the network $\mathcal N$ consists of closed subsets of $X$.
 For every $N\in\mathcal N$ let $(N^{-1}N)^\circ$ be the interior of the set $N^{-1} N=\{x^{-1}y:x,y\in N\}$ in $X$. % Here $\bar N$ stands for the closure of the set $N$ in $X$.

We claim that the family
$$\mathcal B:=\{(N^{-1}N)^\circ:N\in\mathcal N,\;\;e\in(N^{-1}N)^\circ\}$$is a countable neighborhood base at the unit $e$ of the group $X$. Given any neighborhood $U\subset X$ of $e$, find an open neighborhood $V\subset X$ of $e$ such that $V^{-1}V\subset U$. Since $\mathcal N$ is a network for $X$, the open set $V$ coincides with the union $\bigcup\mathcal N_V$ of the subfamily $\mathcal N_V:=\{N\in\mathcal N:N\subset V\}$. Since $V$ is Baire (as an open subspace of the Baire space $X$), some set $N\in\mathcal N_V$ is not nowhere dense in $V$. Consequently, its closure $\bar N=N$ contains a non-empty open subset $W$ of $X$. Then $W^{-1}W$ is an open neighborhood of $e$ such that $W^{-1}W\subset N^{-1}{N}\subset V^{-1}V\subset U$. Consequently, $W^{-1}W\subset (N^{-1}N)^\circ$ and $(N^{-1}N)\in\mathcal B$.

Being first-countable, the topological group $X$ is metrizable by the Birkhoff-Kakutani Theorem \cite[9.1]{Ke}. The metrizable space $X$ is separable as it has a countable network.
\end{proof}

\begin{lemma}\label{l:G1} Each analytic Baire topological group is Polish.
\end{lemma}

\begin{proof}  Being a continuous image of a Polish space, the space $X$ has a countable network \cite[4.9]{Grue}. Applying Lemma~\ref{l:Bm}, we conclude that the topological group $X$ is metrizable and separable. Then its Raikov completion $\bar H$ is a Polish group. By the Lusin-Sierpi\'nski Theorem \cite[21.6]{Ke}, the analytic subspace $X$ of the Polish space $\bar X$ has the Baire property in $\bar X$ and hence $X$ contains a $G_\delta$-subset $G$ of $\bar X$  such that $X\setminus G$ is meager in $\bar X$. Since $X$ is Baire, the $G_\delta$-subset $G$ is dense in $X$ and in $\bar X$. Assuming that $X\ne\bar X$, we can choose a point $\bar x\in\bar X\setminus X$ and conclude that $G\subset X$ and $G\bar x\subset \bar X$ are two disjoint dense $G_\delta$-sets in the Polish space $\bar X$, which contradicts the Baire Theorem. This contradiction shows that $X=\bar X$ and hence $X$ is a Polish group.
\end{proof}

\begin{lemma}\label{l:G2} Each banalytic cosmic Baire topological group is Polish.
\end{lemma}

\begin{proof} Assume that a topological group $X$ is Baire, banalytic and cosmic. By Corollary~\ref{c:ban}, the space $X$ is analytic and by Lemma~\ref{l:G1}, the topological group $X$ is Polish.
\end{proof}

\begin{lemma}\label{l:G3} Each banalytic Baire topological group with countable pseudocharacter is Polish.
\end{lemma}

\begin{proof} Assume that a topological group $X$ is banalytic, Baire, and has a countable pseudocharacter. By Theorem~\ref{t:spread}, the topological group $X$ has countable spread and hence is $\w$-narrow. By Guran's Theorem \cite{Guran} (see also \cite[3.4.23]{AT}), the $\w$-narrow topological group $X$ can be identified with a subgroup of a Tychonoff product $\prod_{\alpha\in A}X_\alpha$ of metrizable separable topological groups $X_\alpha$. Since $X$ has countable pseudocharacter, there exists a countable subset $C\subset A$ such that the projection $\pr_C:X\to X_C\subset\prod_{\alpha\in C}X_\alpha$ is injective. By  Theorem~\ref{t:fac}, the set $C$ is contained in a countable subset $D\subset A$ such that the projection $X_D$ is a Baire space. It follows from  the injectivity of $\pr_C$ and the inclusion $C\subset D$ that the projection $\pr_D:X\to X_D$ is a continuous bijective map. The space $X_D$ is banalytic, being a continuous image of the banalytic space $X$. Since the topological group $X_D\subset\prod_{\alpha\in D}X_\alpha$ is second-countable, we can apply Lemma~\ref{l:G2} and conclude that the topological group $X_D$ is Polish.

 It remains to prove that the bijective continuous map $\pr_D:X\to X_D$ is a homeomorphism. Since $\pr_D$ is a group homomorphism, it suffices to check that for any neighborhood $U\subset X$ of the unit $e\in X$, the image $\pr_D(U)$ is a neighborhood of the unit in the Polish group $X_D$. Using the continuity of the group operations on $X$, find an open neighborhood $V\subset X$ of $e$ such that $VV^{-1}\subset U$. Since $X$ is $\w$-narrow, there exists a countable set $L\subset X$ such that $X=L\cdot V$. Then $X_D=\pr_D(L)\cdot\pr_D(V)$, which implies that the set $W:=\pr_D(V)$ is not meager in the Polish group $X$.

By the banalycity of $X$, there exists a Borel surjective map $f:P\to X$, defined on a Polish space. By Lemma~\ref{l:Bn}, we can additionally assume that the map $\pr_D\circ f:P\to X_D$ is continuous.  Since the map $f$ is Borel, the preimage $f^{-1}(V)$ is a Borel subset of the Polish space $X$. By \cite[14.4]{Ke}, the continuous image $W=\pr_D(V)=\pr_D\circ f(f^{-1}(V))$ of the Borel set $f^{-1}(V)\subset P$ is an analytic space. By the  Lusin-Sierpi\'nski Theorem \cite[21.6]{Ke}, the analytic subset $W$ has the Baire Property in $X$. Since $W$ is a non-meager set with the Baire Property in the Polish  group $X_D$, we can apply the Pettis Theorem \cite[9.9]{Ke} and conclude that $WW^{-1}=\pr_D(VV^{-1})\subset \pr_D(U)$ is a neighborhood of the unit in the Polish group $X_D$. Therefore, the homomorphism $\pr_D:X\to X_D$ is open and hence is a homeomorphism. Consequently, the topological group $X$ is Polish.
\end{proof}

\begin{lemma}\label{l:G4} Under PFA each banalytic Baire topological group $X$ is Polish.
\end{lemma}

\begin{proof} By Corollary~\ref{c:PFA}, under PFA the banalytic Hausdorff space $X$ has countable pseudocharacter. By Lemma~\ref{l:G3}, the topological group $X$ is Polish.
\end{proof}

\section{Proof of Theorem~\ref{t:S}}\label{s:S}

In this section we prove Theorem~\ref{t:S}. The non-trivial implications $(2,3,4)\Ra(1)$ of this theorem are deduced from Theorem~\ref{t:G} and Lemmas~\ref{l:S2} and \ref{l:S3}, proved below.

\begin{lemma}\label{l:S2} Each cosmic banalytic Baire semitopological group $X$ is a Polish topological group.
\end{lemma}

\begin{proof}  By Corollary~\ref{c:ban}, the cosmic banalytic space $X$ is analytic. Being cosmic, the space $X$ is normal \cite[3.8.2]{Eng} and hence Tychonoff. Now we see that $X$ is an analytic Baire Tychonoff semitopological group. By Theorem~3.3 \cite{Bouziad} of Bouziad, $X$ is a topological group. Now we can apply Theorem~\ref{t:G} and conclude that $X$ is a Polish topological group.
\end{proof}

\begin{lemma}\label{l:S3}  A semitopological group $X$ is a topological group if $X$ is banalytic, Baire, and $X$ is topologically isomorphic to a subgroup of the Tychonoff product of cosmic semitopological groups.
\end{lemma}

\begin{proof} By our assumption, we can identify $X$ with a subgroup of the Tychonoff product $\prod_{\alpha\in A}X_\alpha$ of cosmic semitopological groups $X_\alpha$.

Let $\mathcal P$ be the family of all countable subsets $D\subset A$ such that the projection $X_D$ of $X$ onto the $D$-face $\prod_{\alpha\in D}X_\alpha$ of $\prod_{\alpha\in A}X_\alpha$ is a Polish topological group.

We claim that each countable subset $C\subset A$ is contained in some set $D\in\mathcal P$. Indeed, by the factorization Theorem~\ref{t:fac}, every countable subset $C\subset A$ is contained in a countable set $D\subset A$ such that the projection $X_D$ of $X$ onto $\prod_{\alpha\in D}X_\alpha$  is a Baire space. Taking into account that the projection $\pr_D:X\to X_D$ is continuous and the space $X$ is banalytic, we conclude that the cosmic space $X_D$ is banalytic. By Lemma~\ref{l:S2}, the banalytic cosmic Baire semitopological group $X_D$ is a Polish topological group, which means that $D\in\mathcal P$.

Taking into account that each finite subset $C\subset A$ is contained in a set $D\in\mathcal P$, we can show that the map
$$p:X\to \prod_{D\in\mathcal P}X_D,\;\;p:x\mapsto(\pr_D(x))_{D\in\mathcal P},$$
is a topological embedding, which implies that $X$ is a topological group.
\end{proof}

\section{Proof of Theorem~\ref{t:P}}\label{s:P}

 Theorem~\ref{t:P} can be deduced from Theorems~\ref{t:G}, \ref{t:S}, \ref{t:spread} and
 Lemma~\ref{l:P1}, \ref{l:P2} providing conditions under which a paratopological group is a topological group.
These lemmas are taken from the Ph.D. Thesis of Ravsky \cite[Lemmas 5.1, 5.9, 5.10]{RDiss} and were obtained by Ravsky in collaboration with Reznichenko, see for instance, ~ \cite[Proposition 4]{RavRez}.

\begin{lemma}\label{l:P1} A paratopological group $X$ is a topological group if and only if for any neighborhood $U\subset X$ of the unit $e$ the set $\overline{U}\cap\overline{U^{-1}}$ has non-empty interior in $X$.
\end{lemma}

\begin{proof} Only the ``if'' part needs a proof. For a subset $A$ of $X$ by $A^\circ$ we denote the interior of $A$ in $X$. By our assumption, for every neighborhood $U\subset X$ of $e$ we have $\overline U\cap \overline{U^{-1}}^\circ\supset (\overline{U}\cap\overline{U^{-1}})^\circ
\ne\varnothing$ and hence $U\cap (U^{-1})^{2\circ}\supset U\cap \overline{U^{-1}}^\circ \ne\varnothing$.
Then $e\in (U^{-1})^{2\circ}U^{-1}\subset (U^{-1})^{3\circ}$.
%Then $e\in (U^2\cap (U^{-1})^4)^\circ\subset (U^{-1})^{4\circ}$.
Thus $G$ is a topological group.
\end{proof}

\begin{lemma}\label{l:P2} A Baire paratopological group $X$ is a topological group if one of the following conditions is satisfied:
\begin{enumerate}
\item $X$ is totally $\w$-narrow;
\item the square $X\times X$ has countable spread;
\item $X$ is a $T_1$-space and the square $X\times X$ has countable
extent.
\end{enumerate}
\end{lemma}

\begin{proof} 
1. If $X$ is totally $\w$-narrow, then for every neighborhood $U\subset X$ of the unit there exists a countable set $C\subset X$ such that $X=C\cdot (U\cap U^{-1})$. Since $X$ is Baire, the set $\overline{U\cap U^{-1}}$ has non-empty interior in $X$ and by Lemma~\ref{l:P1}, $X$ is a topological group.
\smallskip

2,3. It is easy to check that the subgroup
$X^\sharp=\{(x,x^{-1}): x\in X\}$ of the paratopological group $X\times X$ is a topological group, which is $\w$-narrow if  the space $X^\sharp$ has countable spread or $X^\sharp$ is $T_1$ and has countable extent, see \cite[3.4.7]{AT}). The $\w$-narrowness of the topological group $X^\sharp$ implies the total $\w$-narowness of the paratopological group $X$. By the first item, $X$ is a topological group.
\end{proof}

\section{Acknowledgment}

The authors express their sincere thanks to Ramiro de la Vega for the reference to the CH-example \cite{HJ} of Hajnal and Juh\'asz that gives an answer to the question \cite{MO}, posed by the first author at {\tt MathOverflow}.
%\newpage

\end{document}